\documentclass{amsart}
\usepackage{amssymb,amsmath,amsthm}
\usepackage[dvips]{graphicx}
\usepackage{color}

\let\oldmarginpar\marginpar
\renewcommand\marginpar[1]{\-\oldmarginpar[\raggedleft\footnotesize #1]%
{\raggedright\footnotesize #1}}

\begin{document}

\newtheorem{theorem}{Theorem}[section]
\newtheorem{corollary}[theorem]{Corollary}
\newtheorem{lemma}[theorem]{Lemma}
\newtheorem{proposition}[theorem]{Proposition}
\theoremstyle{definition}
\newtheorem{definition}[theorem]{Definition}
\theoremstyle{remark}
\newtheorem{remark}[theorem]{Remark}
\theoremstyle{definition}
\newtheorem{example}[theorem]{Example}

\def\rank{{\text{rank}\,}}

\numberwithin{equation}{section}

\title[Almost h-conformal slant submersions]{Almost h-conformal slant submersions from almost quaternionic Hermitian manifolds}

\author{Kwang-Soon Park}
\address{KP: Division of General Mathematics, Room 4-107, Changgong Hall, University of Seoul, Seoul 02504, Republic of Korea}
\email{parkksn@gmail.com}

\author{JeongHyeong Park}
\address{JHP: Department of Mathematics, Sungkyunkwan University, Suwon, 16419, Republic of Korea}
\email{parkj@skku.edu}

\keywords{horizontally conformal submersion, quaternionic manifold,
totally geodesic}

\subjclass[2000]{53C15, 53C26, 53C43.}   

\begin{abstract}
We introduce the notions of h-conformal slant submersions and almost
h-conformal slant submersions from almost quaternionic Hermitian
manifolds onto Riemannian manifolds as a generalization of
Riemannian submersions, horizontally conformal submersions, slant
submersions, h-slant submersions, almost h-slant submersion and
conformal slant submersions.

We investigate several properties of these including the
integrability of distributions, the geometry of foliations and the
conditions for such maps to be totally geodesic. Further we give
some examples of such maps.
\end{abstract}

\maketitle
\section{Introduction}\label{intro}
\addcontentsline{toc}{section}{Introduction}

{ One of the most significant works on Riemannian submersions was
introduced by B. O'Neill \cite{O}
in 1966.
 Building on this work, B. Watson \cite{W} introduced the notion of almost Hermitian submersions in 1976.
 In this study, he investigated several differential geometric properties between base manifolds and total manifolds as well as fibers.

In addition, during the 1970s, as a generalization of  Riemannian
submersions, B. Fuglede \cite{F} and T. Ishihara \cite{I} introduced
a horizontally conformal submersion.

Later in 1997, S. Gudmundsson and J. C. Wood \cite{GW} developed
conformal holomorphic submersions between almost Hermitian
manifolds. And they found the condition for a conformal holomorphic
submersion to be a harmonic morphism.

This work was further developed by B. Sahin \cite{S} when he defined
a slant submersion from an almost Hermitian manifold onto a
Riemannian manifold in 2011. This development led him to obtaining
several properties including the integrability of distributions, the
geometry of foliations, the conditions for such maps to be harmonic
or totally geodesic as well as decomposition theorems.

Related work was conducted by Park \cite{P} in 2012, with the
generalization of a slant submersion from an almost Hermitian
manifold, introducing the concept of an h-slant submersion and an
almost h-slant submersion from an almost quaternionic Hermitian
manifold.

In 2017, M. A. Akyol and B. Sahin \cite{AS} defined a conformal
slant submersion from an almost Hermitian manifold onto a Riemannian
manifold.

If we consider a Riemannian submersion
 with some additional
conditions, we get many different types of submersions: An almost
Hermitian submersion \cite{W}, a slant submersion \cite{C, S},
a semi-slant submersion\cite{PP}, a quaternionic submersion
\cite{IMV}, an anti-invariant quaternionic submersion and an
anti-invariant octonion submersion \cite{GIP},
 a horizontally conformal
submersion \cite{G2, BW}, a conformal anti-invariant submersion
\cite{AS2}, an h-conformal semi-invariant submersion and an almost
h-conformal semi-invariant submersion \cite{P4}. We also refer
\cite{G,  P2, S2, LPSS, P3}.


The Riemannian submersions have had several applications in physics
throughout the years. These have included use as inputs to the
Yang-Mills theory \cite{BL2, W2}, Kaluza-Klein theory \cite{IV, BL},
supergravity and superstring theories \cite{IV2, M}. Furthermore,
the quaternionic K\"{a}hler manifolds have been applied
for nonlinear $\sigma-$models with supersymmetry \cite{CMMS}.

This paper has been structured as follows. In section 2 we review
several notions, which are required in the following sections. In
section 3 we introduce the notions of h-conformal slant submersions
and almost h-conformal slant submersions and investigate several of
their properties: the geometry of foliations, the integrability of
distributions, the harmonicity of such maps and the conditions for
such maps to be totally geodesic and some others.

In the final section 4, we provide some examples of h-conformal
slant submersions and almost h-conformal slant submersions.

\section{Preliminaries}\label{prelim}

In this section we review several important concepts required for
the following sections.

%
%
%

{{ Let $(M, g_M)$ and $(N, g_N)$ be Riemannian manifolds.
 We say that a smooth surjective
map $F$ from $(M, g_M)$ to $(N, g_N)$ is a {\em Riemannian submersion} if \\
(i) the differential $(F_*)_p$  is surjective
for any $p\in M$,\\ 
(ii) for any $p\in M$, $(F_*)_p$ is an isometry between $(\ker
(F_*)_p)^{\perp}$ and $T_{F(p)} N$,
 where $(\ker (F_*)_p)^{\perp}$ is
the orthogonal complement of the space $\ker (F_*)_p$ in the tangent
space $T_p M$ of $M$ at $p$ \cite{FIP, GLP,  O}.}}

The map $F$ is said to be {\em horizontally weakly conformal}
\cite{BW} if it satisfies either (i) $(F_*)_p = 0$ or (ii) $(F_*)_p$
is surjective and there exists a positive number $\lambda (p) > 0$
such that
\begin{equation}\label{eq1}
g_N ((F_*)_p X, (F_*)_p Y) = \lambda^2 g_M (X, Y) \quad \text{for} \
X,Y\in (\ker (F_*)_p)^{\perp},
\end{equation}
at any point $p\in M$. In the event that the map is applicable only
at a certain point $p$, then it is called {\em horizontally weakly
conformal} at $p$. If it holds with type (i) then the point $p$ is
called a {\em critical point}, similarly if it holds with type (ii),
we call the point $p$ a {\em regular point}. Further, the positive
number $\lambda (p)$ is called a {\em dilation} of $F$ at $p$.
{{A horizontally weakly conformal map $F$ is said to be a {\em horizontally conformal submersion} if $F$ has no critical
points.}}

Let $F : (M,g_M) \mapsto (N,g_N)$ be a horizontally conformal
submersion.


Let $\mathcal{V}$ be the vertical projection (distribution) and
$\mathcal{H}$ the horizontal projection (distribution).

The (O'Neill) tensors $\mathcal{T}$ and $\mathcal{A}$ were defined
by
\begin{eqnarray}
  \mathcal{A}_E F &=& \mathcal{H}\nabla_{\mathcal{H}E} \mathcal{V}F+\mathcal{V}\nabla_{\mathcal{H}E} \mathcal{H}F  \label{eq3} \\
   \mathcal{T}_E F &=& \mathcal{H}\nabla_{\mathcal{V}E} \mathcal{V}F+\mathcal{V}\nabla_{\mathcal{V}E}
   \mathcal{H}F  \label{eq4}
\end{eqnarray}
for  $E, F\in \Gamma(TM)$, where $\nabla$ is the Levi-Civita
connection of $g_M$ (\cite{O, FIP}). Then we easily obtain
 that
\begin{equation}\label{eqn22}
g_M (\mathcal{T}_U V, W) = -g_M (V, \mathcal{T}_U W)
\end{equation}
\begin{equation}\label{eqn23}
g_M (\mathcal{A}_U V, W) = -g_M (V, \mathcal{A}_U W)
\end{equation}
for $U,V,W\in \Gamma(TM)$.

We define $\widehat{\nabla}_U V := \mathcal{V}\nabla_U V$ for
$U,V\in \Gamma(\ker F_*)$.

Let $F : (M, g_M) \mapsto (N, g_N)$ be a smooth map. Then the {\em
second fundamental form} of $F$ is provided by
$$
(\nabla F_*)(X,Y) := \nabla^F _X F_* Y-F_* (\nabla _XY) \quad
\text{for} \ X,Y\in \Gamma(TM),
$$
where $\nabla^F$ is the pullback connection and $\nabla$ is the
Levi-Civita connection of the metrics $g_M$ and $g_N$ \cite{BW}.

We note that if the tension field $\tau(F) := trace (\nabla F_*)=0$,
then $F$ is called a {\em harmonic} map and if $(\nabla F_*)(X,Y)=0$
for $X,Y\in \Gamma (TM)$ then $F$ is called a {\em totally geodesic}
map  \cite{BW}.

\begin{lemma}\label{lem1} \cite{U}
Let $(M, g_M)$ and $(N, g_N)$ be Riemannian manifolds and $F :
(M,g_M)\mapsto (N,g_N)$ a smooth map. Then we obtain
\begin{equation}\label{eq5}
\nabla_X^F F_*Y - \nabla_Y^F F_*X - F_*([X,Y]) = 0
\end{equation}
for $X,Y\in \Gamma(TM)$.
\end{lemma}

\begin{remark}
 By \eqref{eq5}, the second fundamental form
$\nabla F_*$ is shown to be symmetric. Additionally we obtain
\begin{equation}\label{eqn5}
[V,X]\in \Gamma(\ker F_*)
\end{equation}
for $V\in \Gamma(\ker F_*)$ and $X\in \Gamma((\ker F_*)^{\perp})$.
\end{remark}

\begin{proposition}\label{prop1}\cite{G2}
Let  $F : (M, g_M) \mapsto (N, g_N)$ be a horizontally conformal
submersion with dilation $\lambda$. Then we get
\begin{equation}\label{eq7}
\mathcal{A}_X Y = \frac{1}{2}\{\mathcal{V}[X,Y] - \lambda^2 g_M (X,
Y) \nabla_{\mathcal{V}} (\frac{1}{\lambda^2})\},
\end{equation}
for $X,Y\in \Gamma((\ker F_*)^{\perp})$,
where $\nabla_{\mathcal{V}}$ is the gradient vector field
in the vertical distribution $\mathcal{V} \subset TM$ which is
expressed by
$\displaystyle{\nabla_{\mathcal{V}} f = \sum_{i=1}^{m} W_i(f) W_i}$
for $f\in C^{\infty}(M)$ and a local orthonormal frame $\{ W_1,
\cdots, W_m \}$ of $\mathcal{V}$, where $m = \dim \mathcal{V}$.
\end{proposition}

\begin{proposition}\label{prop2}
Let  $F : (M, g_M) \mapsto (N, g_N)$ be a horizontally conformal
submersion.  Then we obtain
\begin{equation}\label{eq07}
\mathcal{T}_V W = \mathcal{T}_W V
\end{equation}
for $V,W\in \Gamma(\ker F_*)$.
\end{proposition}

\begin{proof}
Given $V,W\in \Gamma(\ker F_*)$ and $X\in \Gamma((\ker
F_*)^{\perp})$, we obtain
\begin{align*}
g_M(\mathcal{T}_V W, X)
  &= g_M(\mathcal{H}\nabla_V W, X) = g_M(\mathcal{H}(\nabla_W V + [V,W]), X)    \\
  &= g_M(\mathcal{H}\nabla_W V, X) = g_M(\mathcal{T}_W V, X),
\end{align*}
which implies (\ref{eq07}).
\end{proof}

\begin{remark}
Considering (\ref{eq7}), for a horizontally conformal submersion
$F$, $\mathcal{A}_X Y = -\mathcal{A}_Y X$ generally does not hold
for $X,Y\in \Gamma((\ker F_*)^{\perp})$.
\end{remark}

\begin{lemma}\label{lem1}\cite{BW}
Let  $F : (M, g_M) \mapsto (N, g_N)$ be a horizontally conformal
submersion with dilation $\lambda$. Then we get
\begin{eqnarray}
 (\nabla F_*)(X,Y) &=& X(\ln \lambda) F_*Y + Y(\ln \lambda) F_*X - g_M (X, Y) F_* (\nabla \ln \lambda)   \label{eqn77}
\end{eqnarray}
for $X,Y\in \Gamma((\ker F_*)^{\perp})$.
\end{lemma}

Let $(M, g_M, J)$ be an almost Hermitian manifold, where $J$ is an
almost complex structure on $M$.
A horizontally conformal submersion $F : (M,g_M,J)\mapsto (N,g_N)$
is called a {\em conformal slant submersion} \cite{AS} if the angle
$\theta = \theta (X)$ between $JX$ and the space $\ker (F_*)_p$ is
constant for nonzero $X\in \ker (F_*)_p$ and $p\in M$.

{{Let $M$ be a $C^{\infty}$-manifold of dimension $4m$ and let $E$
be a subbundle of $\text{End} (TM)$ such that given $p\in M$ with a
neighborhood $U$, there exists a local basis $\{ I,J,K \}$ of
sections of $E$ on $U$ such that
\begin{equation}\label{hypercom}
I^2 = J^2 = K^2 = -id, \quad IJ = -JI = K.
\end{equation}
Then the subbundle $E$ is said to be an {\em almost quaternionic
structure} on $M$ and $(M,E)$ an {\em almost quaternionic manifold}
\cite{AM}.
Additionally, if we have a Riemannian metric $g$ on $(M, E)$ such
that
\begin{equation}\label{hypermet}
g(RX, RY)=g(X, Y)
\end{equation}
for $R\in \{ I,J,K \}$ and $X, Y\in \Gamma(TM)$, then we call
$(M,E,g)$ an {\em almost quaternionic Hermitian manifold}
\cite{IMV}. The basis $\{ I,J,K \}$ satisfying (\ref{hypercom}) and
(\ref{hypermet}) is called a {\em quaternionic Hermitian basis}.

An almost quaternionic Hermitian manifold $(M,E,g_M)$ is called a
{\em quaternionic K\"{a}hler manifold} if the subbundle $E$ is
preserved by the Levi-Civita connection $\nabla$ of the metric $g_M$
(\cite{I0, IMV}).

If we have a global quaternionic Hermitian basis $\{ I,J,K \}$ of
sections of $E$ on $M$ such that $R$ is parallel with respect to the
Levi-Civita connection $\nabla$ of the metric $g$ for any $R\in \{
I,J,K \}$, then we call $(M, E, g )$ a {\em hyperk\"{a}hler
manifold}, $(I, J, K, g )$ a {\em hyperk\"{a}hler structure} on $M$,
and $g$ a {\em hyperk\"{a}hler metric} \cite{B}.}}

Let $(M, E, g_M)$ be an almost quaternionic Hermitian manifold and
$(N, g_N)$ a Riemannian manifold. A Riemannian submersion $F : (M,
E, g_M) \mapsto (N, g_N)$ is said to be an {\em almost h-slant
submersion} if given a point $p\in M$ with a neighborhood $U$, there
exists a
 quaternionic Hermitian basis $\{ I,J,K \}$ of sections of
$E$ on $U$ such that given $R\in \{ I,J,K \}$, the angle $\theta_R =
\theta_R(X)$ between $RX$ and the space $\ker (F_*)_q$ is constant
for nonzero $X\in \ker (F_*)_q$ and $q\in U$ \cite{P}.

We call such a basis $\{ I,J,K \}$ an {\em almost h-slant basis}.

Moreover, an almost h-slant submersion $F : (M, E, g_M) \mapsto (N,
g_N)$ is called an {\em h-slant submersion} if $\theta=\theta_I
=\theta_J =\theta_K$ \cite{P}.

We call such a basis $\{ I,J,K \}$ an {\em h-slant basis} and the
angle $\theta$ an {\em h-slant angle}.

Throughout this paper, we will use the above notations.


\section{Almost h-conformal slant submersions}\label{slant}

In this section, we introduce the notions of h-conformal slant
submersions and almost h-conformal slant submersions from almost
quaternionic Hermitian manifolds onto Riemannian manifolds.
 We will study the integrability of distributions and
 the geometry of foliations. And we investigate the harmonicity of such maps
 and the conditions for such maps to be totally geodesic.

\begin{definition}
Let $(N,g_N)$  be a Riemannian manifold and  $(M,E,g_M)$ an almost
quaternionic Hermitian manifold. Let $F : (M,E,g_M)\mapsto (N,g_N)$
be a horizontally conformal submersion. We call the map $F$ an {\em
almost h-conformal slant submersion} if given a point $p\in M$ with
a neighborhood $U$, there exists a  quaternionic Hermitian basis $\{
I,J,K \}$ of sections of $E$ on $U$ such that given $R\in \{ I,J,K
\}$, the angle $\theta_R = \theta_R(X)$ between $RX$ and the space
$\ker (F_*)_q$ is constant for nonzero $X\in \ker (F_*)_q$ and $q\in
U$.
\end{definition}

We call such a basis $\{ I,J,K \}$ an {\em almost h-conformal slant
basis} and the angles $\theta_I$, $\theta_J$, $\theta_K$ {\em slant
angles}.

\begin{definition}
Let $(N,g_N)$  be a Riemannian manifold and  $(M,E,g_M)$ an almost
quaternionic Hermitian manifold. Let $F : (M,E,g_M)\mapsto (N,g_N)$
be an almost h-conformal slant submersion. We call the map $F$ an
{\em h-conformal slant submersion} if $\theta = \theta_I = \theta_J
= \theta_K$, where $\theta_I$, $\theta_J$, $\theta_K$ are slant
angles.
\end{definition}

We call such a basis $\{ I,J,K \}$ an {\em h-conformal slant basis}
and the angle $\theta$ {\em h-slant angle}.

Let $F$ be an almost h-conformal slant submersion from an almost
quaternionic Hermitian manifold $(M,E,g_M)$ onto a Riemannian
manifold $(N,g_N)$. Given a point $p\in M$ with a neighborhood $U$,
we get an almost h-conformal slant basis $\{ I,J,K \}$ of sections
of $E$ on $U$.

Then given $V\in \Gamma(\ker F_*)$ and $R\in \{ I,J,K \}$, we have
\begin{equation}\label{eq08}
RV = \phi_R V + \omega_R V,
\end{equation}
where $\phi_R V\in \Gamma(\ker F_*)$ and $\omega_R V\in \Gamma((\ker
F_*)^{\perp})$.

Given $X\in \Gamma((\ker F_*)^{\perp})$ and $R\in \{ I,J,K \}$, we
write
\begin{equation}\label{eq8}
RX = B_R X + C_R X,
\end{equation}
where $B_R X\in \Gamma(\ker F_*)$ and $C_R X\in \Gamma((\ker
F_*)^{\perp})$.

Given $R\in \{ I,J,K \}$, we obtain the orthogonal decomposition
\begin{equation}\label{eq09}
(\ker F_*)^{\perp} = \omega_R (\ker F_*) \oplus \mu^R.
\end{equation}
We can easily check that $\mu^R$ is $R$-invariant for $R\in \{ I,J,K
\}$.

Furthermore, given $V\in \Gamma(\ker F_*)$, $X\in \Gamma((\ker
F_*)^{\perp})$, and $R\in \{ I,J,K \}$, from $R^2 = -id$, we get
\begin{equation}\label{eq09-1}
\phi_R^2 V + B_R \omega_R V = -V,
\end{equation}
\begin{equation}\label{eq09-2}
\omega_R\phi_R V + C_R \omega_R V = 0,
\end{equation}
\begin{equation}\label{eq09-3}
\phi_R B_R X + B_R C_R X = 0,
\end{equation}
\begin{equation}\label{eq09-4}
\omega_R B_R X + C_R^2 X = -X.
\end{equation}
And we also have
\begin{align*}
R\omega_R V
  &= B_R\omega_R V + C_R\omega_R V = -V - \phi_R^2 V - \omega_R \phi_R V    \\
  &= -\sin^2 \theta_R V - \omega_R \phi_R V.
\end{align*}

Then we easily obtain the following:

\begin{lemma}\label{lem04}
Let $(M,I,J,K,g_M)$ be a hyperk\"{a}hler manifold and $(N, g_N)$ a
Riemannian manifold. Let $F : (M,I,J,K,g_M) \mapsto (N, g_N)$ be an
almost h-conformal slant submersion with $(I,J,K)$ an almost
h-conformal slant basis. Then we have

\begin{enumerate}
\item
\begin{align*}
  &\widehat{\nabla}_V \phi_R W + \mathcal{T}_V \omega_R W = \phi_R \widehat{\nabla}_V W + B_R \mathcal{T}_V W    \\
  &\mathcal{T}_V \phi_R W + \mathcal{H}\nabla_V \omega_R W = C_R \mathcal{T}_V W + \omega_R \widehat{\nabla}_V W
\end{align*}
for $V,W\in \Gamma(\ker F_*)$ and $R\in \{ I,J,K \}$.
\item
\begin{align*}
  &\mathcal{A}_X C_R Y + \mathcal{V}\nabla_X B_R Y = \phi_R \mathcal{A}_X Y + B_R \mathcal{H} \nabla_X Y    \\
  &\mathcal{H} \nabla_X C_R Y + \mathcal{A}_X B_R Y = \omega_R\mathcal{A}_X Y + C_R \mathcal{H}\nabla_X Y
\end{align*}
for $X,Y\in \Gamma((\ker F_*)^{\perp})$ and $R\in \{ I,J,K \}$.
\item
\begin{align*}
  &\mathcal{A}_X \omega_R V + \mathcal{V}\nabla_X \phi_R V = B_R \mathcal{A}_X V + \phi_R \mathcal{V}\nabla_X V  \\
  &\mathcal{H}\nabla_X \omega_R V + \mathcal{A}_X \phi_R V = C_R \mathcal{A}_X V + \omega_R\mathcal{V}\nabla_X V
\end{align*}
for $V\in \Gamma(\ker F_*)$, $X\in \Gamma((\ker F_*)^{\perp})$, and
$R\in \{ I,J,K \}$.
\end{enumerate}
\end{lemma}

We define
\begin{align*}
  &(\nabla_V \phi_R)(W) := \widehat{\nabla}_V \phi_R W - \phi_R \widehat{\nabla}_V W     \\
  &(\nabla_V \omega_R)(W) := \mathcal{H} \nabla_V \omega_R W -
  \omega_R \widehat{\nabla}_V W
\end{align*}
for $V,W\in \Gamma(\ker F_*)$ and $R\in \{ I,J,K \}$.

By Lemma \ref{lem04}, we have

\begin{equation}\label{eq10}
(\nabla_V \phi_R)(W) = B_R \mathcal{T}_V W - \mathcal{T}_V \omega_R
W
\end{equation}
\begin{equation}\label{eq010}
(\nabla_V \omega_R)(W) = C_R \mathcal{T}_V W - \mathcal{T}_V \phi_R
W
\end{equation}

We call $\phi_R$ and $\omega_R$ {\em parallel} if $\nabla \phi_R =
0$ and $\nabla \omega_R = 0$, respectively.

\begin{proposition}
Let $(M,E,g_M)$ be an almost quaternionic Hermitian manifold and
$(N, g_N)$ a Riemannian manifold. Let $F : (M,E,g_M) \mapsto (N,
g_N)$ be an almost h-conformal slant submersion. Then we have
\begin{equation}\label{eq11}
\phi_R^2 V = -\cos^2 \theta_R V
\end{equation}
for $V\in \Gamma(\ker F_*)$ and $R\in \{ I,J,K \}$, where $(I,J,K)$
is an almost h-conformal slant basis with the slant angles $
\theta_I,\theta_J,\theta_K$.
\end{proposition}

\begin{proof}
If $V = 0$, then the result immediately follows.

Given a nonzero vector $V\in \ker (F_*)_p$, $p\in M$, we obtain
\begin{align*}
\cos \theta_R
  &= \sup_{W\in \ker (F_*)_p -\{ 0 \}} \frac{g_M(RV, W)}{|RV|\cdot |W|} = \sup_{W\in \ker (F_*)_p -\{ 0 \}} \frac{g_M(\phi_R V, W)}{|V|\cdot |W|}    \\
  &= \sup_{W\in \ker (F_*)_p -\{ 0 \}} \frac{g_M(\phi_R V, \frac{W}{|W|})}{|V|} = \frac{g_M(\phi_R V, \frac{\phi_R V}{|\phi_R V|})}{|V|}   \\
  &= \frac{|\phi_R V|}{|V|}
\end{align*}
so that
$$
\cos^2 \theta_R g_M(V, V) = g_M(\phi_R V, \phi_R V) = g_M(-\phi_R^2
V, V).
$$
By polarization, we get
$$
\cos^2 \theta_R g_M(V+W, V+W) = g_M(-\phi_R^2 (V+W), V+W)
$$
for $W\in \ker (F_*)_p$, which implies
$$
cos^2 \theta_R g_M(V, W) = g_M(-\phi_R^2 V, W).
$$
Hence,
$$
g_M(\phi_R^2 V + cos^2 \theta_R V, W) = 0,
$$
it shows the result.
\end{proof}

\begin{remark}
Let $F$ be an almost h-conformal slant submersion from an almost
quaternionic Hermitian manifold $(M,E,g_M)$ onto a Riemannian
manifold $(N,g_N)$. Then given an almost h-conformal slant basis
$(I,J,K)$ with the slant angles $\theta_I,\theta_J,\theta_K$, we
easily have
\begin{equation}\label{eq12}
g_M(\phi_R V, \phi_R W) = \cos^2 \theta_R g_M(V, W),
\end{equation}
\begin{equation}\label{eq13}
g_M(\omega_R V, \omega_R W) = \sin^2 \theta_R g_M(V, W),
\end{equation}
for $V,W\in \Gamma(\ker F_*)$ and $R\in \{ I,J,K \}$.
\end{remark}

\begin{lemma}
Let $(M,I,J,K,g_M)$ be a hyperk\"{a}hler manifold and $(N, g_N)$ a
Riemannian manifold. Let $F : (M,I,J,K,g_M) \mapsto (N, g_N)$ be an
almost h-conformal slant submersion with $(I,J,K)$ an almost
h-conformal slant basis. Assume that $\omega_R$ is parallel for
$R\in \{ I,J,K \}$. Then we get
\begin{equation}\label{eq14-1}
\mathcal{T}_{\phi_R V} \phi_R V = -\cos^2 \theta_R \mathcal{T}_V V
\end{equation}
for $V\in \Gamma(\ker F_*)$.
\end{lemma}

\begin{proof}
Since $\omega_R$ is parallel for $R\in \{ I,J,K \}$, by
(\ref{eq010}), we have
\begin{equation}\label{eq015-1}
C_R \mathcal{T}_V W = \mathcal{T}_V \phi_R W
\end{equation}
for $V,W\in \Gamma(\ker F_*)$.

Given $V\in \Gamma(\ker F_*)$ and $X\in \Gamma((\ker F_*)^{\perp})$,
by using (\ref{eq015-1}) and (\ref{eq07}),  we obtain
\begin{align*}
g_M(X, \mathcal{T}_{\phi_R V} \phi_R V)
  &= g_M(X, C_R \mathcal{T}_{\phi_R V} V) = g_M(X, C_R \mathcal{T}_{V} \phi_R V)    \\
  &= g_M(X, \mathcal{T}_{V} \phi_R^2 V) = g_M(X, -\cos^2 \theta_R \mathcal{T}_{V}
  V),
\end{align*}
which implies the result.
\end{proof}

\begin{theorem}\label{thm1-2}
Let $(M,I,J,K,g_M)$ be a hyperk\"{a}hler manifold and $(N, g_N)$ a
Riemannian manifold. Let $F : (M,I,J,K,g_M) \mapsto (N, g_N)$ be an
almost h-conformal slant submersion with $(I,J,K)$ an almost
h-conformal slant basis. Assume that $\omega_R$ is parallel with the
slant angle $0\leq \theta_R < \frac{\pi}{2}$ for some $R\in \{ I,J,K
\}$. Then all the fibers of the map $F$ are minimal.
\end{theorem}

\begin{proof}
We may assume that $\omega_I$ is parallel with the slant angle
$0\leq \theta_I < \frac{\pi}{2}$.

By (\ref{eq11}), we can choose a local orthonormal frame
$\{e_i\}_{i=1}^{2n}$  of $\ker F_*$ such that $e_{2i} = \sec
\theta_I \phi_I e_{2i-1}$ for $1\leq i \leq n$.

Then by (\ref{eq14-1}), we have
\begin{align*}
2nH
  &= \sum_{i=1}^{2n} \mathcal{T}_{e_i} e_i    \\
  &= \sum_{i=1}^{n} (\mathcal{T}_{e_{2i-1}} e_{2i-1} + \mathcal{T}_{\sec \theta_I \phi_I e_{2i-1}} \sec \theta_I \phi_I e_{2i-1})    \\
  &= \sum_{i=1}^{n} (\mathcal{T}_{e_{2i-1}} e_{2i-1} + \sec^2 \theta_I\cdot (-\cos^2 \theta_I) \mathcal{T}_{e_{2i-1}} e_{2i-1})    \\
  &= 0.
\end{align*}
Therefore, the result follows.
\end{proof}

Now, we study the integrability of distributions and the geometry of
foliations.

\begin{theorem}\label{int1}
Let $(M,I,J,K,g_M)$ be a hyperk\"{a}hler manifold and $(N, g_N)$ a
Riemannian manifold. Let $F : (M,I,J,K,g_M) \mapsto (N, g_N)$ be an
almost h-conformal slant submersion with $(I,J,K)$ an almost
h-conformal slant basis. Then the following conditions are
equivalent:

(a) the distribution $(\ker F_*)^{\perp}$ is integrable.

(b)
\begin{align*}
  &\lambda^{-2} g_N(\nabla_Y^F F_* C_I X - \nabla_X^F F_* C_I Y, F_* \omega_I V)  \\
  &= g_M(\mathcal{V}\nabla_X B_I Y + \mathcal{A}_X C_I Y - \mathcal{V}\nabla_Y B_I X - \mathcal{A}_Y C_I X, \phi_I V)   \\
  &+ g_M(\mathcal{A}_X B_I Y - \mathcal{A}_Y B_I X - X(\ln \lambda)C_I Y + Y(\ln \lambda)C_I X - C_I Y(\ln \lambda)X \\
  &+ C_I X(\ln \lambda)Y + 2g_M(X, C_I Y)(\nabla \ln \lambda), \omega_I V)
\end{align*}
for $X,Y\in \Gamma((\ker F_*)^{\perp})$ and $V\in \Gamma(\ker F_*)$.

(c)
\begin{align*}
  &\lambda^{-2} g_N(\nabla_Y^F F_* C_J X - \nabla_X^F F_* C_J Y, F_* \omega_J V)  \\
  &= g_M(\mathcal{V}\nabla_X B_J Y + \mathcal{A}_X C_J Y - \mathcal{V}\nabla_Y B_J X - \mathcal{A}_Y C_J X, \phi_J V)   \\
  &+ g_M(\mathcal{A}_X B_J Y - \mathcal{A}_Y B_J X - X(\ln \lambda)C_J Y + Y(\ln \lambda)C_J X - C_J Y(\ln \lambda)X \\
  &+ C_J X(\ln \lambda)Y + 2g_M(X, C_J Y)(\nabla \ln \lambda), \omega_J V)
\end{align*}
for $X,Y\in \Gamma((\ker F_*)^{\perp})$ and $V\in \Gamma(\ker F_*)$.

(d)
\begin{align*}
  &\lambda^{-2} g_N(\nabla_Y^F F_* C_K X - \nabla_X^F F_* C_K Y, F_* \omega_K V)  \\
  &= g_M(\mathcal{V}\nabla_X B_K Y + \mathcal{A}_X C_K Y - \mathcal{V}\nabla_Y B_K X - \mathcal{A}_Y C_K X, \phi_K V)   \\
  &+ g_M(\mathcal{A}_X B_K Y - \mathcal{A}_Y B_K X - X(\ln \lambda)C_K Y + Y(\ln \lambda)C_K X - C_K Y(\ln \lambda)X \\
  &+ C_K X(\ln \lambda)Y + 2g_M(X, C_K Y)(\nabla \ln \lambda), \omega_K V)
\end{align*}
for $X,Y\in \Gamma((\ker F_*)^{\perp})$ and $V\in \Gamma(\ker F_*)$.
\end{theorem}

\begin{proof}
Given $X,Y\in \Gamma((\ker F_*)^{\perp})$, $V\in \Gamma(\ker F_*)$,
and $R\in \{ I,J,K \}$, by using (\ref{eqn77}), we get
\begin{align*}
  &g_M([X,Y], V)    \\
  &= g_M(\nabla_X RY - \nabla_Y RX, RV)    \\
  &= g_M(\mathcal{V}\nabla_X B_R Y + \mathcal{A}_X C_R Y - \mathcal{V}\nabla_Y B_R X - \mathcal{A}_Y C_R X, \phi_R V)    \\
  &+ g_M(\mathcal{A}_X B_R Y + \mathcal{H}\nabla_X C_R Y - \mathcal{A}_Y B_R X - \mathcal{H}\nabla_Y C_R X, \omega_R V)   \\
  &= g_M(\mathcal{V}\nabla_X B_R Y + \mathcal{A}_X C_R Y - \mathcal{V}\nabla_Y B_R X - \mathcal{A}_Y C_R X, \phi_R V)    \\
  &+ g_M(\mathcal{A}_X B_R Y - \mathcal{A}_Y B_R X, \omega_R V) + \lambda^{-2} g_N(F_* \nabla_X C_R Y - F_* \nabla_Y C_R X, F_* \omega_R V)   \\
  &= g_M(\mathcal{V}\nabla_X B_R Y + \mathcal{A}_X C_R Y - \mathcal{V}\nabla_Y B_R X - \mathcal{A}_Y C_R X, \phi_R V)    \\
  &+ g_M(\mathcal{A}_X B_R Y - \mathcal{A}_Y B_R X, \omega_R V)  \\
  &+ \lambda^{-2} g_N(-(\nabla F_*)(X,C_R Y) + \nabla_X^F F_* C_R Y + (\nabla F_*)(Y,C_R X) - \nabla_Y^F F_* C_R X, F_* \omega_R V)   \\
  &= g_M(\mathcal{V}\nabla_X B_R Y + \mathcal{A}_X C_R Y - \mathcal{V}\nabla_Y B_R X - \mathcal{A}_Y C_R X, \phi_R V)    \\
  &+ g_M(\mathcal{A}_X B_R Y - \mathcal{A}_Y B_R X, \omega_R V) + \lambda^{-2} g_N(\nabla_X^F F_* C_R Y - \nabla_Y^F F_* C_R X, F_* \omega_R V) \\
  &+ \lambda^{-2} g_N(-X(\ln \lambda)F_* C_R Y - C_R Y(\ln \lambda)F_* X + g_M(X, C_R Y)F_* (\nabla \ln \lambda)  \\
  &+ Y(\ln \lambda)F_* C_R X + C_R X(\ln \lambda)F_* Y - g_M(Y, C_R X)F_* (\nabla \ln \lambda), F_* \omega_R V)   \\
  &= g_M(\mathcal{V}\nabla_X B_R Y + \mathcal{A}_X C_R Y - \mathcal{V}\nabla_Y B_R X - \mathcal{A}_Y C_R X, \phi_R V)    \\
  &+ g_M(\mathcal{A}_X B_R Y - \mathcal{A}_Y B_R X - X(\ln \lambda)C_R Y + Y(\ln \lambda)C_R X - C_R Y(\ln \lambda)X  \\
  &+ C_R X(\ln \lambda)Y + 2g_M(X, C_R Y)(\nabla \ln \lambda), \omega_R V)   \\
  &+ \lambda^{-2} g_N(\nabla_X^F F_* C_R Y - \nabla_Y^F F_* C_R X, F_* \omega_R V).
\end{align*}
Hence, $(a) \Leftrightarrow (b)$, $(a) \Leftrightarrow (c)$, $(a)
\Leftrightarrow (d)$.

Therefore, the result follows.
\end{proof}

We deal with the condition for an almost h-conformal slant
submersion to be horizontally homothetic.

\begin{theorem}\label{int2}
Let $(M,I,J,K,g_M)$ be a hyperk\"{a}hler manifold and $(N, g_N)$ a
Riemannian manifold. Let $F : (M,I,J,K,g_M) \mapsto (N, g_N)$ be an
almost h-conformal slant submersion with $(I,J,K)$ an almost
h-conformal slant basis. Assume that the distribution $(\ker
F_*)^{\perp}$ is integrable. Then the following conditions are
equivalent:

(a) the map $F$ is horizontally homothetic.

(b)
\begin{align*}
  &\lambda^{-2} g_N (\nabla_Y^F F_* C_I X - \nabla_X^F F_* C_I Y, F_* \omega_I V)  \\
  &= g_M(\mathcal{V}\nabla_X B_I Y + \mathcal{A}_X C_I Y - \mathcal{V}\nabla_Y B_I X - \mathcal{A}_Y C_I X, \phi_I V)   \\
  &+ g_M (\mathcal{A}_X B_I Y - \mathcal{A}_Y B_I X, \omega_I V)
\end{align*}
for $X,Y\in \Gamma((\ker F_*)^{\perp})$ and $V\in \Gamma(\ker F_*)$.

(c)
\begin{align*}
  &\lambda^{-2} g_N (\nabla_Y^F F_* C_J X - \nabla_X^F F_* C_J Y, F_* \omega_J V)  \\
  &= g_M(\mathcal{V}\nabla_X B_J Y + \mathcal{A}_X C_J Y - \mathcal{V}\nabla_Y B_J X - \mathcal{A}_Y C_J X, \phi_J V)   \\
  &+ g_M (\mathcal{A}_X B_J Y - \mathcal{A}_Y B_J X, \omega_J V)
\end{align*}
for $X,Y\in \Gamma((\ker F_*)^{\perp})$ and $V\in \Gamma(\ker F_*)$.

(d)
\begin{align*}
  &\lambda^{-2} g_N (\nabla_Y^F F_* C_K X - \nabla_X^F F_* C_K Y, F_* \omega_K V)  \\
  &= g_M(\mathcal{V}\nabla_X B_K Y + \mathcal{A}_X C_K Y - \mathcal{V}\nabla_Y B_K X - \mathcal{A}_Y C_K X, \phi_K V)   \\
  &+ g_M (\mathcal{A}_X B_K Y - \mathcal{A}_Y B_K X, \omega_K V)
\end{align*}
for $X,Y\in \Gamma((\ker F_*)^{\perp})$ and $V\in \Gamma(\ker F_*)$.
\end{theorem}

\begin{proof}
Given $X,Y\in \Gamma((\ker F_*)^{\perp})$, $V\in \Gamma(\ker F_*)$,
and $R\in \{ I,J,K \}$, by the assumption, we obtain
\begin{eqnarray}
  &&0 = g_M ([X,Y], V) \label{eq112} \\
  &&=g_M(\mathcal{V}\nabla_X B_R Y + \mathcal{A}_X C_R Y - \mathcal{V}\nabla_Y B_R X - \mathcal{A}_Y C_R X, \phi_R V)  \nonumber   \\
  &&+ g_M(\mathcal{A}_X B_R Y - \mathcal{A}_Y B_R X - X(\ln \lambda)C_R Y + Y(\ln \lambda)C_R X - C_R Y(\ln \lambda)X  \nonumber \\
  &&+ C_R X(\ln \lambda)Y + 2g_M(X, C_R Y)(\nabla \ln \lambda), \omega_R V) \nonumber \\
  &&+ \lambda^{-2} g_N(\nabla_X^F F_* C_R Y - \nabla_Y^F F_* C_R X, F_* \omega_R V). \nonumber
\end{eqnarray}
Using (\ref{eq112}), we have $(a) \Rightarrow (b)$, $(a) \Rightarrow
(c)$, $(a) \Rightarrow (d)$.

Conversely, from (\ref{eq112}), we get
\begin{eqnarray}
  &&0 = g_M (- X(\ln \lambda)C_R Y + Y(\ln \lambda)C_R X - C_R Y(\ln \lambda)X + C_R X(\ln \lambda)Y \label{eq113} \\
  &&+ 2g_M(X, C_R Y)(\nabla \ln \lambda), \omega_R V). \nonumber
\end{eqnarray}
If $Y\in \Gamma(\mu^R)$, then by using (\ref{eq113}) and
(\ref{eq09}), we have
\begin{equation}\label{eq114}
0 = g_M (Y(\ln \lambda)C_R X - R Y(\ln \lambda)X + 2g_M(X, R
Y)(\nabla \ln \lambda), \omega_R V).
\end{equation}
Applying $X=RY$ at (\ref{eq114}), we get
\begin{eqnarray}
  &&0 = g_M (Y(\ln \lambda)R^2 Y - R Y(\ln \lambda)RY + 2g_M(RY, R
Y)(\nabla \ln \lambda), \omega_R V) \label{eq115} \\
  &&= 2g_M(Y, Y) g_M(\nabla \ln \lambda, \omega_R V), \nonumber
\end{eqnarray}
which implies
\begin{equation}\label{eq116}
g_M (\nabla \lambda, \omega_R V) = 0 \quad \text{for} \ V\in
\Gamma(\ker F_*).
\end{equation}
Applying $X = \omega_R V$ at (\ref{eq114}), we obtain
\begin{align*}
  0
  &= g_M (Y(\ln \lambda)C_R \omega_R V - R Y(\ln \lambda)\omega_R V, \omega_R V)   \\
  &= - R Y(\ln \lambda) g_M (\omega_R V, \omega_R V),
\end{align*}
which means
\begin{equation}\label{eq117}
g_M (\nabla \lambda, Z) = 0 \quad \text{for} \ Z\in \Gamma(\mu^R).
\end{equation}
By (\ref{eq116}) and (\ref{eq117}), we have $(b) \Rightarrow (a)$,
$(c) \Rightarrow (a)$, $(d) \Rightarrow (a)$.

Therefore, the result follows.
\end{proof}

\begin{theorem}\label{thm12}
Let $(M,I,J,K,g_M)$ be a hyperk\"{a}hler manifold and $(N, g_N)$ a
Riemannian manifold. Let $F : (M,I,J,K,g_M) \mapsto (N, g_N)$ be an
almost h-conformal slant submersion with $(I,J,K)$ an almost
h-conformal slant basis. Then the following conditions are
equivalent:

(a) the distribution $(\ker F_*)^{\perp}$ defines a totally geodesic
foliation on $M$.

(b)
\begin{align*}
&\lambda^{-2}g_N(\nabla_X^F F_* Y, F_* \omega_I \phi_I V) -
\lambda^{-2}g_N(\nabla_X^F F_* C_I Y, F_* \omega_I V) \\
&= g_M(\mathcal{A}_X B_I Y, \omega_I V)  \\
&+ g_M(-X(\ln \lambda) C_I Y - C_I Y(\ln \lambda) X + g_M(X, C_I Y)
(\nabla \ln \lambda), \omega_I V) \\
&- g_M(-X(\ln \lambda) Y - Y(\ln \lambda) X + g_M(X, Y) (\nabla \ln
\lambda), \omega_I \phi_I V)
\end{align*}
for $X,Y\in \Gamma((\ker F_*)^{\perp})$ and $V\in \Gamma(\ker F_*)$.

(c)
\begin{align*}
&\lambda^{-2}g_N(\nabla_X^F F_* Y, F_* \omega_J \phi_J V) -
\lambda^{-2}g_N(\nabla_X^F F_* C_J Y, F_* \omega_J V) \\
&= g_M(\mathcal{A}_X B_J Y, \omega_J V)  \\
&+ g_M(-X(\ln \lambda) C_J Y - C_J Y(\ln \lambda) X + g_M(X, C_J Y)
(\nabla \ln \lambda), \omega_J V) \\
&- g_M(-X(\ln \lambda) Y - Y(\ln \lambda) X + g_M(X, Y) (\nabla \ln
\lambda), \omega_J \phi_J V)
\end{align*}
for $X,Y\in \Gamma((\ker F_*)^{\perp})$ and $V\in \Gamma(\ker F_*)$.

(d)
\begin{align*}
&\lambda^{-2}g_N(\nabla_X^F F_* Y, F_* \omega_K \phi_K V) -
\lambda^{-2}g_N(\nabla_X^F F_* C_K Y, F_* \omega_K V) \\
&= g_M(\mathcal{A}_X B_K Y, \omega_K V)  \\
&+ g_M(-X(\ln \lambda) C_K Y - C_K Y(\ln \lambda) X + g_M(X, C_K Y)
(\nabla \ln \lambda), \omega_K V) \\
&- g_M(-X(\ln \lambda) Y - Y(\ln \lambda) X + g_M(X, Y) (\nabla \ln
\lambda), \omega_K \phi_K V)
\end{align*}
for $X,Y\in \Gamma((\ker F_*)^{\perp})$ and $V\in \Gamma(\ker F_*)$.
\end{theorem}

\begin{proof}
Given $X,Y\in \Gamma((\ker F_*)^{\perp})$, $V\in \Gamma(\ker F_*)$,
and $R\in \{ I,J,K \}$, we obtain
\begin{align*}
  &g_M (\nabla_X Y, V)  \\
  &= g_M (\nabla_X RY, RV)   \\
  &= g_M (\nabla_X RY, \phi_R V + \omega_R V)   \\
  &= \cos^2 \theta_R \cdot g_M(\nabla_X Y, V) - g_M(\nabla_X Y, \omega_R
  \phi_R V) + g_M(\mathcal{A}_X B_R Y + \mathcal{H}\nabla_X C_R Y, \omega_R V)
\end{align*}
so that
\begin{align*}
  &\sin^2 \theta_R \cdot g_M (\nabla_X Y, V)    \\
  &= - g_M(\nabla_X Y, \omega_R \phi_R V) + g_M(\mathcal{A}_X B_R Y, \omega_R V) + \lambda^{-2}g_N(F_* \nabla_X C_R Y, F_* \omega_R V)  \\
  &= - g_M(\nabla_X Y, \omega_R \phi_R V) + g_M(\mathcal{A}_X B_R Y, \omega_R V)   \\
  &+ \lambda^{-2}g_N(-(\nabla F_*)(X,C_R Y) + \nabla_X^F F_* C_R Y, F_* \omega_R V) \\
  &= - g_M(\nabla_X Y, \omega_R \phi_R V) + g_M(\mathcal{A}_X B_R Y, \omega_R V) + \lambda^{-2}g_N(\nabla_X^F F_* C_R Y, F_* \omega_R V)  \\
  &+ \lambda^{-2}g_N(-X(\ln \lambda) F_* C_R Y - C_R Y(\ln \lambda) F_* X + g_M(X, C_R Y) F_* (\nabla \ln \lambda), F_* \omega_R V) \\
  &= - g_M(\nabla_X Y, \omega_R \phi_R V) + g_M(\mathcal{A}_X B_R Y, \omega_R V) + \lambda^{-2}g_N(\nabla_X^F F_* C_R Y, F_* \omega_R V)  \\
  &+ g_M(-X(\ln \lambda) C_R Y - C_R Y(\ln \lambda) X + g_M(X, C_R Y) (\nabla \ln \lambda), \omega_R V). \\
\end{align*}
But
\begin{align*}
  &g_M(\nabla_X Y, \omega_R \phi_R V)    \\
  &= \lambda^{-2}g_N(F_* \nabla_X Y, F_* \omega_R \phi_R V)   \\
  &= \lambda^{-2}g_N(-(\nabla F_*)(X,Y) + \nabla_X^F F_* Y, F_* \omega_R \phi_R V)   \\
  &= \lambda^{-2}g_N(-X(\ln \lambda) F_* Y - Y(\ln \lambda) F_* X + g_M(X, Y)F_* (\nabla \ln \lambda) \\
  &+ \nabla_X^F F_* Y, F_* \omega_R \phi_R V) \\
  &= g_M(-X(\ln \lambda) Y - Y(\ln \lambda) X + g_M(X, Y) (\nabla \ln \lambda), \omega_R \phi_R V)  \\
  &+ \lambda^{-2}g_N(\nabla_X^F F_* Y, F_* \omega_R \phi_R V).
\end{align*}
Hence, we get  $(a) \Leftrightarrow (b)$, $(a) \Leftrightarrow (c)$,
$(a) \Leftrightarrow (d)$.

Therefore, the result follows.
\end{proof}

\begin{theorem}\label{thm22}
Let $(M,I,J,K,g_M)$ be a hyperk\"{a}hler manifold and $(N, g_N)$ a
Riemannian manifold. Let $F : (M,I,J,K,g_M) \mapsto (N, g_N)$ be an
almost h-conformal slant submersion with $(I,J,K)$ an almost
h-conformal slant basis. Assume that the distribution $(\ker
F_*)^{\perp}$ defines a totally geodesic foliation on $M$. Then the
following conditions are equivalent:

(a) the map $F$ is horizontally homothetic.

(b)
\begin{align*}
&\lambda^{-2}g_N(\nabla_X^F F_* Y, F_* \omega_I \phi_I V) -
\lambda^{-2}g_N(\nabla_X^F F_* C_I Y, F_* \omega_I V) \\
&= g_M(\mathcal{A}_X B_I Y, \omega_I V)
\end{align*}
for $X,Y\in \Gamma((\ker F_*)^{\perp})$ and $V\in \Gamma(\ker F_*)$.

(c)
\begin{align*}
&\lambda^{-2}g_N(\nabla_X^F F_* Y, F_* \omega_J \phi_J V) -
\lambda^{-2}g_N(\nabla_X^F F_* C_J Y, F_* \omega_J V) \\
&= g_M(\mathcal{A}_X B_J Y, \omega_J V)
\end{align*}
for $X,Y\in \Gamma((\ker F_*)^{\perp})$ and $V\in \Gamma(\ker F_*)$.

(d)
\begin{align*}
&\lambda^{-2}g_N(\nabla_X^F F_* Y, F_* \omega_K \phi_K V) -
\lambda^{-2}g_N(\nabla_X^F F_* C_K Y, F_* \omega_K V) \\
&= g_M(\mathcal{A}_X B_K Y, \omega_K V)
\end{align*}
for $X,Y\in \Gamma((\ker F_*)^{\perp})$ and $V\in \Gamma(\ker F_*)$.
\end{theorem}

\begin{proof}
Given $X,Y\in \Gamma((\ker F_*)^{\perp})$, $V\in \Gamma(\ker F_*)$,
and $R\in \{ I,J,K \}$, by Theorem {{\ref{thm12}}}, we obtain
\begin{eqnarray}
  &&\lambda^{-2}g_N(\nabla_X^F F_* Y, F_* \omega_R\phi_R V) -
\lambda^{-2}g_N(\nabla_X^F F_* C_R Y, F_* \omega_R V) \label{eq1-15} \\
  &&= g_M(\mathcal{A}_X B_R Y, \omega_R V) \nonumber  \\
  &&+ g_M(-X(\ln \lambda) C_R Y - C_R Y(\ln \lambda) X + g_M(X, C_R Y)
(\nabla \ln \lambda), \omega_R V) \nonumber \\
&&- g_M(-X(\ln \lambda) Y - Y(\ln \lambda) X + g_M(X, Y) (\nabla \ln
\lambda), \omega_R \phi_R V),  \nonumber
\end{eqnarray}
which implies $(a) \Rightarrow (b)$, $(a) \Rightarrow (c)$, $(a)
\Rightarrow (d)$.

Conversely, from (\ref{eq1-15}), we get
\begin{eqnarray}
&&0 = g_M(-X(\ln \lambda) C_R Y - C_R Y(\ln \lambda) X + g_M(X, C_R
Y) (\nabla \ln \lambda), \omega_R V)  \label{eq015}  \\
&&- g_M(-X(\ln \lambda) Y - Y(\ln \lambda) X + g_M(X, Y) (\nabla \ln
\lambda), \omega_R \phi_R V).  \nonumber
\end{eqnarray}

If $Y\in \Gamma(\mu^R)$, then from (\ref{eq015}), we have
\begin{eqnarray}
&&0 = g_M(- R Y(\ln \lambda) X + g_M(X, R
Y) (\nabla \ln \lambda), \omega_R V)  \label{eq015-2}  \\
&&- g_M(- Y(\ln \lambda) X + g_M(X, Y) (\nabla \ln \lambda),
\omega_R \phi_R V).  \nonumber
\end{eqnarray}
Applying $X = RY$ at (\ref{eq015-2}), we obtain
\begin{equation}\label{eq0151-3}
0 = g_M(RX, RY) g_M(\nabla \ln \lambda, \omega_R V),
\end{equation}
which implies
\begin{equation}\label{eq0151-4}
g_M(\nabla \lambda, \omega_R V) = 0 \quad \text{for} \ V\in
\Gamma(\ker F_*).
\end{equation}
Applying $X = \omega_R V$ at (\ref{eq015-2}), by using (\ref{eq09})
and (\ref{eq13}), we get
\begin{eqnarray}
&&0 = -g_M(\omega_R V, \omega_R V) \cdot g_M(RY, \nabla \ln \lambda) + g_M(Y, \nabla \ln \lambda) \cdot g_M(\omega_R V, \omega_R \phi_R V) \label{eq015-5}  \\
&&= -g_M(\omega_R V, \omega_R V) \cdot g_M(RY, \nabla \ln \lambda),
\nonumber
\end{eqnarray}
which implies
\begin{equation}\label{eq0151-6}
g_M(\nabla \lambda, Z) = 0 \quad \text{for} \ Z\in \Gamma(\mu^R).
\end{equation}
By (\ref{eq0151-4}) and (\ref{eq0151-6}), we have $(b) \Rightarrow
(a)$, $(c) \Rightarrow (a)$, $(d) \Rightarrow (a)$.

Therefore, the result follows.
\end{proof}

\begin{theorem}\label{thm002}
Let $(M,I,J,K,g_M)$ be a hyperk\"{a}hler manifold and $(N, g_N)$ a
Riemannian manifold. Let $F : (M,I,J,K,g_M) \mapsto (N, g_N)$ be an
almost h-conformal slant submersion with $(I,J,K)$ an almost
h-conformal slant basis. Then the following conditions are
equivalent:

(a) the distribution $\ker F_*$ defines a totally geodesic foliation
on $M$.

(b)
$$
g_M(\nabla_V \omega_I \phi_I W, X) = g_M(\mathcal{T}_V \omega_I W,
B_I X) + g_M(\mathcal{H} \nabla_V \omega_I W, C_I X)
$$
for $X\in \Gamma((\ker F_*)^{\perp})$ and $V,W\in \Gamma(\ker F_*)$.

(c)
$$
g_M(\nabla_V \omega_J \phi_J W, X) = g_M(\mathcal{T}_V \omega_J W,
B_J X) + g_M(\mathcal{H} \nabla_V \omega_J W, C_J X)
$$
for $X\in \Gamma((\ker F_*)^{\perp})$ and $V,W\in \Gamma(\ker F_*)$.

(d)
$$
g_M(\nabla_V \omega_K \phi_K W, X) = g_M(\mathcal{T}_V \omega_K W,
B_K X) + g_M(\mathcal{H} \nabla_V \omega_K W, C_K X)
$$
for $X\in \Gamma((\ker F_*)^{\perp})$ and $V,W\in \Gamma(\ker F_*)$.
\end{theorem}

\begin{proof}
Given $V,W\in \Gamma(\ker F_*)$, $X\in \Gamma((\ker F_*)^{\perp})$,
and $R\in \{ I,J,K \}$, we have
\begin{align*}
&g_M(\nabla_V W, X) \\
&= g_M(\nabla_V RW, RX)    \\
&= g_M(\nabla_V (\phi_R W + \omega_R W), RX)    \\
&= -g_M(\nabla_V \phi_R^2 W + \nabla_V \omega_R \phi_R W, X) +  g_M(\nabla_V \omega_R W, B_R X + C_R X)   \\
&= \cos^2 \theta_R g_M(\nabla_V W, X) -  g_M(\nabla_V \omega_R \phi_R W, X)   \\
&+ g_M(\mathcal{T}_V \omega_R W, B_R X) + g_M(\mathcal{H} \nabla_V
\omega_R W, C_R X)
\end{align*}
so that
\begin{align*}
\sin^2 \theta_R g_M(\nabla_V W, X) =& -  g_M(\nabla_V \omega_R
\phi_R W, X) + g_M(\mathcal{T}_V \omega_R W, B_R X)  \\
& + g_M(\mathcal{H} \nabla_V \omega_R W, C_R X).
\end{align*}
Hence, we get $(a) \Leftrightarrow (b)$, $(a) \Leftrightarrow (c)$,
$(a) \Leftrightarrow (d)$.

Therefore, we obtain the result.
\end{proof}

We recall the following:

\begin{lemma}\cite{BW}\label{lem03}
Let $(M,g_M)$ and $(N,g_N)$ be Riemannian manifolds. Let $F :
(M,g_M) \mapsto (N,g_N)$ be a horizontally conformal submersion with
dilation $\lambda$.

Then we have
\begin{equation}\label{eq28}
\tau(F) = -mF_* H + (2-n)F_* (\nabla \ln \lambda),
\end{equation}
where $\tau(F)$ and $H$ denote the tension field of $F$ and the mean
curvature vector field of  $\ker F_*$, respectively, $m = \dim \ker
F_*$, $n = \dim N$.
\end{lemma}

Using Lemma \ref{lem03} and Theorem \ref{thm1-2}, we obtain

\begin{corollary}
Let $(M,I,J,K,g_M)$ be a hyperk\"{a}hler manifold and $(N, g_N)$ a
Riemannian manifold. Let $F : (M,I,J,K,g_M) \mapsto (N, g_N)$ be an
almost h-conformal slant submersion with $(I,J,K)$ an almost
h-conformal slant basis and $\dim N > 2$. Assume that $\omega_R$ is
parallel with the slant angle $0\leq \theta_R < \frac{\pi}{2}$ for
some $R\in \{ I,J,K \}$. Then the following conditions are
equivalent:

(a) the map $F$ is harmonic.

(b) the map $F$ is horizontally homothetic.
\end{corollary}

\begin{corollary}
Let $(M,I,J,K,g_M)$ be a hyperk\"{a}hler manifold and $(N, g_N)$ a
Riemannian manifold. Let $F : (M,I,J,K,g_M) \mapsto (N, g_N)$ be an
almost h-conformal slant submersion with $(I,J,K)$ an almost
h-conformal slant basis and $\dim N = 2$. Assume that $\omega_R$ is
parallel with the slant angle $0\leq \theta_R < \frac{\pi}{2}$ for
some $R\in \{ I,J,K \}$. Then the map $F$ is harmonic.
\end{corollary}

\begin{definition}
Let $(M,I,J,K,g_M)$ be a hyperk\"{a}hler manifold and $(N, g_N)$ a
Riemannian manifold. Let $F : (M,I,J,K,g_M) \mapsto (N, g_N)$ be an
almost h-conformal slant submersion with $(I,J,K)$ an almost
h-conformal slant basis. Then given $R\in \{ I,J,K \}$, we call the
map $F$ {\em
 $(\omega_R\ker F_*, \mu^R)$-totally geodesic} if it satisfies $(\nabla F_*)(\omega_R V,X) = 0$
 for $V\in \Gamma(\ker F_*)$ and $X\in \Gamma(\mu^R)$.
\end{definition}

\begin{theorem}
Let $(M,I,J,K,g_M)$ be a hyperk\"{a}hler manifold and $(N, g_N)$ a
Riemannian manifold. Let $F : (M,I,J,K,g_M) \mapsto (N, g_N)$ be an
almost h-conformal slant submersion with $(I,J,K)$ an almost
h-conformal slant basis. Then the following conditions are
equivalent:

(a) the map $F$ is horizontally homothetic.

(b) the map $F$ is $(\omega_I\ker F_*, \mu^I)$-totally geodesic.

(c) the map $F$ is $(\omega_J\ker F_*, \mu^J)$-totally geodesic.

(d) the map $F$ is $(\omega_K\ker F_*, \mu^K)$-totally geodesic.
\end{theorem}

\begin{proof}
Given $V\in \Gamma(\ker F_*)$, $X\in \Gamma(\mu^R)$, and $R\in \{
I,J,K \}$, by using (\ref{eqn77}), we get
\begin{align*}
(\nabla F_*)(\omega_R V,X) &= \omega_R V(\ln \lambda) F_*X + X(\ln \lambda) F_* \omega_R V - g_M (\omega_R V, X) F_* (\nabla \ln \lambda)    \\
      &= \omega_R V(\ln \lambda) F_*X + X(\ln \lambda) F_* \omega_R V.
\end{align*}
Since $g_N (F_* X, F_* \omega_R V) = \lambda^2 g_M (X, \omega_R V) =
0$, $\{ F_* X, F_* \omega_R V \}$ is linearly independent for
nonzero $X, V$.

Hence, we obtain $(a) \Leftrightarrow (b)$, $(a) \Leftrightarrow
(c)$, $(a) \Leftrightarrow (d)$.

Therefore, the result follows.
\end{proof}

\begin{theorem}\label{thm: 03-4}
Let $(M,I,J,K,g_M)$ be a hyperk\"{a}hler manifold and $(N, g_N)$ a
Riemannian manifold. Let $F : (M,I,J,K,g_M) \mapsto (N, g_N)$ be an
almost h-conformal slant submersion with $(I,J,K)$ an almost
h-conformal slant basis. Then the following conditions are
equivalent:

(a) the map $F$ is a totally geodesic map.

(b) (i) $C_I (\mathcal{T}_V \phi_I W + \mathcal{H}\nabla_V \omega_I
W) + \omega_I (\widehat{\nabla}_V \phi_I W +  \mathcal{T}_V \omega_I
W) = 0$, (ii) $F$ is horizontally homothetic, (iii)
$C_I(\mathcal{A}_X \phi_I V + \mathcal{H}\nabla_X \omega_I V) +
\omega_I(\mathcal{V}\nabla_X \phi_I V + \mathcal{A}_X \omega_I V) =
0$ for $V,W\in \Gamma(\ker F_*)$ and $X\in \Gamma((\ker
F_*)^{\perp})$.

(c) (i) $C_J (\mathcal{T}_V \phi_J W + \mathcal{H}\nabla_V \omega_J
W) + \omega_J (\widehat{\nabla}_V \phi_J W +  \mathcal{T}_V \omega_J
W) = 0$, (ii) $F$ is horizontally homothetic, (iii)
$C_J(\mathcal{A}_X \phi_J V + \mathcal{H}\nabla_X \omega_J V) +
\omega_J(\mathcal{V}\nabla_X \phi_J V + \mathcal{A}_X \omega_J V) =
0$ for $V,W\in \Gamma(\ker F_*)$ and $X\in \Gamma((\ker
F_*)^{\perp})$.

(d) (i) $C_K (\mathcal{T}_V \phi_K W + \mathcal{H}\nabla_V \omega_K
W) + \omega_K (\widehat{\nabla}_V \phi_K W +  \mathcal{T}_V \omega_K
W) = 0$, (ii) $F$ is horizontally homothetic, (iii)
$C_K(\mathcal{A}_X \phi_K V + \mathcal{H}\nabla_X \omega_K V) +
\omega_K(\mathcal{V}\nabla_X \phi_K V + \mathcal{A}_X \omega_K V) =
0$ for $V,W\in \Gamma(\ker F_*)$ and $X\in \Gamma((\ker
F_*)^{\perp})$.
\end{theorem}

\begin{proof}
Given $V,W\in \Gamma(\ker F_*)$ and $R\in \{ I,J,K \}$, we have
\begin{align*}
      &(\nabla F_*)(V, W)  \\
      &= F_* (R\nabla_V RW)    \\
      &= F_* (R(\nabla_V (\phi_R W + \omega_R W)))   \\
      &= F_* (R(\mathcal{T}_V \phi_R W + \widehat{\nabla}_V \phi_R W + \mathcal{T}_V \omega_R W + \mathcal{H}\nabla_V \omega_R W))   \\
      &= F_* (B_R \mathcal{T}_V \phi_R W + C_R \mathcal{T}_V \phi_R W + \phi_R \widehat{\nabla}_V \phi_R W + \omega_R \widehat{\nabla}_V \phi_R W
      + \phi_R \mathcal{T}_V \omega_R W   \\
      &+ \omega_R \mathcal{T}_V \omega_R W + B_R \mathcal{H}\nabla_V \omega_R W + C_R \mathcal{H}\nabla_V \omega_R W)   \\
      &= F_* (C_R \mathcal{T}_V \phi_R W + \omega_R \widehat{\nabla}_V \phi_R W + \omega_R \mathcal{T}_V \omega_R W + C_R \mathcal{H}\nabla_V \omega_R W)
\end{align*}
so that
\begin{eqnarray}
&&(\nabla F_*)(V, W) = 0   \label{eq028-1}  \\
&&\Leftrightarrow C_R (\mathcal{T}_V \phi_R W + \mathcal{H}\nabla_V
\omega_R W) + \omega_R (\widehat{\nabla}_V \phi_R W +  \mathcal{T}_V
\omega_R W) = 0.  \nonumber
\end{eqnarray}
We claim that $F$ is horizontally homothetic if and only if $(\nabla
F_*)(X,Y) = 0$ for $X,Y\in \Gamma((\ker F_*)^{\perp})$.

By (\ref{eqn77}), we get
\begin{equation}\label{eq028}
(\nabla F_*)(X,Y) = X(\ln \lambda) F_*Y + Y(\ln \lambda) F_*X - g_M
(X, Y) F_* (\nabla \ln \lambda)
\end{equation}
for $X,Y\in \Gamma((\ker F_*)^{\perp})$ so that the part from left
to right is obtained.

Conversely, from (\ref{eq028}), we have
\begin{equation}\label{eq029}
0 = X(\ln \lambda) F_*Y + Y(\ln \lambda) F_*X - g_M (X, Y) F_*
(\nabla \ln \lambda).
\end{equation}
Applying $X = Y$ at (\ref{eq029}), we obtain
\begin{equation}\label{eq030}
0 = 2X(\ln \lambda) F_*X - g_M (X, X) F_* (\nabla \ln \lambda).
\end{equation}
Taking the inner product with $F_* X$ at (\ref{eq030}), we get
$$
0 = \lambda^2 g_M (X, X) g_M (X, \nabla \ln \lambda),
$$
which implies the result.

Given $V\in \Gamma(\ker F_*)$ and $X\in \Gamma((\ker F_*)^{\perp})$,
we have
\begin{align*}
&(\nabla F_*)(X, V)   \\
&= F_* (R\nabla_X RV)   \\
&= F_* (R\nabla_X (\phi_R V + \omega_R V))   \\
&= F_* (R(\mathcal{A}_X \phi_R V + \mathcal{V}\nabla_X \phi_R V + \mathcal{A}_X \omega_R V + \mathcal{H}\nabla_X \omega_R V))   \\
&= F_* (C_R\mathcal{A}_X \phi_R V + \omega_R\mathcal{V}\nabla_X
\phi_R V + \omega_R\mathcal{A}_X \omega_R V + C_R\mathcal{H}\nabla_X
\omega_R V)
\end{align*}
so that
\begin{eqnarray}
&&(\nabla F_*)(X, V) = 0   \label{eq028-12}  \\
&&\Leftrightarrow C_R(\mathcal{A}_X \phi_R V + \mathcal{H}\nabla_X
\omega_R V) + \omega_R(\mathcal{V}\nabla_X \phi_R V + \mathcal{A}_X
\omega_R V) = 0.  \nonumber
\end{eqnarray}
Hence, we obtain $(a) \Leftrightarrow (b)$, $(a) \Leftrightarrow
(c)$, $(a) \Leftrightarrow (d)$.

Therefore, the result follows.
\end{proof}

We consider a decomposition theorem. Denote by $M_{\ker F_*}$ and
$M_{(\ker F_*)^{\perp}}$ the integral manifolds of $\ker F_*$ and
$(\ker F_*)^{\perp}$, respectively. Using Theorem \ref{thm12} and
Theorem \ref{thm002}, we get

\begin{theorem}\label{thm02}
Let $(M,I,J,K,g_M)$ be a hyperk\"{a}hler manifold and $(N, g_N)$ a
Riemannian manifold. Let $F : (M,I,J,K,g_M) \mapsto (N, g_N)$ be an
almost h-conformal slant submersion with $(I,J,K)$ an almost
h-conformal slant basis. Then the following conditions are
equivalent:

(a) $(M,g_M)$ is locally a Riemannian product manifold of the form
$M_{(\ker F_*)^{\perp}} \times M_{\ker F_*}$.

(b)
\begin{align*}
&\lambda^{-2}g_N(\nabla_X^F F_* Y, F_* \omega_I \phi_I V) -
\lambda^{-2}g_N(\nabla_X^F F_* C_I Y, F_* \omega_I V) \\
&= g_M(\mathcal{A}_X B_I Y, \omega_I V)  \\
&+ g_M(-X(\ln \lambda) C_I Y - C_I Y(\ln \lambda) X + g_M(X, C_I Y)
(\nabla \ln \lambda), \omega_I V) \\
&- g_M(-X(\ln \lambda) Y - Y(\ln \lambda) X + g_M(X, Y) (\nabla \ln
\lambda), \omega_I \phi_I V),
\end{align*}
$$
g_M(\nabla_V \omega_I \phi_I W, X) = g_M(\mathcal{T}_V \omega_I W,
B_I X) + g_M(\mathcal{H} \nabla_V \omega_I W, C_I X)
$$
for $X,Y\in \Gamma((\ker F_*)^{\perp})$ and $V,W\in \Gamma(\ker
F_*)$.

(c)
\begin{align*}
&\lambda^{-2}g_N(\nabla_X^F F_* Y, F_* \omega_J \phi_J V) -
\lambda^{-2}g_N(\nabla_X^F F_* C_J Y, F_* \omega_J V) \\
&= g_M(\mathcal{A}_X B_J Y, \omega_J V)  \\
&+ g_M(-X(\ln \lambda) C_J Y - C_J Y(\ln \lambda) X + g_M(X, C_J Y)
(\nabla \ln \lambda), \omega_J V) \\
&- g_M(-X(\ln \lambda) Y - Y(\ln \lambda) X + g_M(X, Y) (\nabla \ln
\lambda), \omega_J \phi_J V),
\end{align*}
$$
g_M(\nabla_V \omega_J \phi_J W, X) = g_M(\mathcal{T}_V \omega_J W,
B_J X) + g_M(\mathcal{H} \nabla_V \omega_J W, C_J X)
$$
for $X,Y\in \Gamma((\ker F_*)^{\perp})$ and $V,W\in \Gamma(\ker
F_*)$.

(d)
\begin{align*}
&\lambda^{-2}g_N(\nabla_X^F F_* Y, F_* \omega_K \phi_K V) -
\lambda^{-2}g_N(\nabla_X^F F_* C_K Y, F_* \omega_K V) \\
&= g_M(\mathcal{A}_X B_K Y, \omega_K V)  \\
&+ g_M(-X(\ln \lambda) C_K Y - C_K Y(\ln \lambda) X + g_M(X, C_K Y)
(\nabla \ln \lambda), \omega_K V) \\
&- g_M(-X(\ln \lambda) Y - Y(\ln \lambda) X + g_M(X, Y) (\nabla \ln
\lambda), \omega_K \phi_K V),
\end{align*}
$$
g_M(\nabla_V \omega_K \phi_K W, X) = g_M(\mathcal{T}_V \omega_K W,
B_K X) + g_M(\mathcal{H} \nabla_V \omega_K W, C_K X)
$$
for $X,Y\in \Gamma((\ker F_*)^{\perp})$ and $V,W\in \Gamma(\ker
F_*)$.
\end{theorem}

\section{Examples}

With coordinates $(x_1,x_2,\cdots,x_{4m})$ on the Euclidean space
$\mathbb{R}^{4m}$ , we choose the following complex structures $I, J, K$ on
$\mathbb{R}^{4m}$:
\begin{align*}
  &I(\tfrac{\partial}{\partial x_{4k+1}})=\tfrac{\partial}{\partial x_{4k+2}},
  I(\tfrac{\partial}{\partial x_{4k+2}})=-\tfrac{\partial}{\partial x_{4k+1}},
  I(\tfrac{\partial}{\partial x_{4k+3}})=\tfrac{\partial}{\partial x_{4k+4}},
  I(\tfrac{\partial}{\partial x_{4k+4}})=-\tfrac{\partial}{\partial x_{4k+3}},     \\
  &J(\tfrac{\partial}{\partial x_{4k+1}})=\tfrac{\partial}{\partial x_{4k+3}},
  J(\tfrac{\partial}{\partial x_{4k+2}})=-\tfrac{\partial}{\partial x_{4k+4}},
  J(\tfrac{\partial}{\partial x_{4k+3}})=-\tfrac{\partial}{\partial x_{4k+1}},
  J(\tfrac{\partial}{\partial x_{4k+4}})=\tfrac{\partial}{\partial x_{4k+2}},    \\
  &K(\tfrac{\partial}{\partial x_{4k+1}})=\tfrac{\partial}{\partial x_{4k+4}},
  K(\tfrac{\partial}{\partial x_{4k+2}})=\tfrac{\partial}{\partial x_{4k+3}},
  K(\tfrac{\partial}{\partial x_{4k+3}})=-\tfrac{\partial}{\partial x_{4k+2}},
  K(\tfrac{\partial}{\partial x_{4k+4}})=-\tfrac{\partial}{\partial x_{4k+1}}
\end{align*}
for $k\in \{ 0,1,\cdots,m-1 \}$.

Then we can see that $(I,J,K,g)$ is a hyperk\"{a}hler structure on
$\mathbb{R}^{4m}$, where $g$ stands for the Euclidean metric on
$\mathbb{R}^{4m}$. Throughout this section, we will maintain the same notational conventions.

\begin{example}
Let $\pi : TM \mapsto M$ be the natural projection, where $(M, E,
g)$ is an almost quaternionic Hermitian manifold. Then the map $\pi$
is an h-conformal slant submersion with the h-slant angle $\theta =
0$ and dilation $\lambda = 1$ \cite{IMV}.
\end{example}

\begin{example}
Let $(N,g_N)$ be a $(4m-1)$-dimensional Riemannian manifold and
$(M,E,g_M)$ a $4m$-dimensional almost quaternionic Hermitian
manifold . Let $F : (M,E,g_M) \mapsto (N,g_N)$ be a horizontally
conformal submersion with dilation $\lambda$, where $\lambda : M
\mapsto \mathbb{R}$ is a smooth positive function. Then the map $F$
is an h-conformal slant submersion with the h-slant angle $\theta =
\frac{\pi}{2}$ and dilation $\lambda$.
\end{example}

\begin{example}
Let $A\subset \mathbb{R}^{4m}$ and $B\subset \mathbb{R}^{4m-1}$ be
open domains for a positive integer $m$. Let $F : A \mapsto B$ be a
horizontally conformal submersion with dilation $\lambda$, where
$\lambda : A \mapsto \mathbb{R}$ is a smooth positive function. Then
the map $F$ is an h-conformal slant submersion with the h-slant
angle $\theta = \frac{\pi}{2}$ and dilation $\lambda$.
\end{example}

\begin{example}
Let $(M_1,E_1,g_{M_1})$ and $(M_2,E_2,g_{M_2})$ be almost
quaternionic Hermitian manifold with $\text{dim} M_1 = 4n$ and
$\text{dim} M_2 = 4m$ for positive integers $n$ and $m$. And let
$(N_1,g_1')$ and $(N_2,g_2')$ be Riemannian manifold with
$\text{dim} N_1 = 4n-1$ and $\text{dim} N_2 = 4m-1$. Let $F_i :
(M_i,E_i,g_{M_i}) \mapsto (N_i,g_i')$ be a Riemannian submersion for
$i\in \{ 1,2 \}$. Consider the map $F : M_1\times M_2 \mapsto
N_1\times N_2$ given by
$$
F(x,y)= \lambda(F_1(x), F_2(y)) \quad \text{for} \ x\in M_1 \
\text{and} \ y\in M_2.
$$
Then the map $F$ is an h-conformal slant submersion with the h-slant
angle $\theta = \frac{\pi}{2}$ and dilation $\lambda$.
\end{example}

\begin{example}\label{exam1}
We define a map $F : \mathbb{R}^4 \mapsto \mathbb{R}^2$ by
$$
F(x_1,\cdots,x_4) = e^{2}(x_4, x_3).
$$
Then the map $F$ is an almost h-conformal slant submersion with the
slant angles $\{ \theta_I = 0, \theta_J = \frac{\pi}{2}, \theta_K =
\frac{\pi}{2} \}$ and dilation $\lambda = e^{2}$.
\end{example}

\begin{example}
We define a map $F : \mathbb{R}^8 \mapsto \mathbb{R}^4$ by
$$
F(x_1,\cdots,x_8) = e^{4}(\frac{x_1 - x_4}{\sqrt{2}}, \frac{x_5 -
x_8}{\sqrt{2}}, x_7, x_2).
$$
Then the map $F$ is an almost h-conformal slant submersion with the
slant angles $\{ \theta_I = \frac{\pi}{4}, \theta_J = \frac{\pi}{4},
\theta_K = \frac{\pi}{2} \}$ and dilation $\lambda = e^{4}$.
\end{example}

\begin{example}
We define a map $F : \mathbb{R}^{4} \mapsto \mathbb{R}^2$ by
$$
F(x_1,\cdots,x_{4}) = e^5(x_1 \cos \alpha - x_3 \sin \alpha, x_2
\sin \beta - x_4 \cos \beta).
$$
Then the map $F$ is an almost h-conformal slant submersion with the
slant angles $\{ \theta_I, \theta_J = \frac{\pi}{2}, \theta_K \}$
and dilation $\lambda = e^{5}$ such that $\cos \theta_I = |\sin
(\alpha + \beta)|$ and $\cos \theta_K = |\cos (\alpha + \beta)|$.
\end{example}

\begin{example}
We define a map $F : \mathbb{R}^{4} \mapsto \mathbb{R}^2$ by
$$
F(x_1,\cdots,x_{4}) = e^7(x_1, \frac{\sqrt{3}}{2} x_2 - \frac{1}{2}
x_4).
$$
Then the map $F$ is an almost h-conformal slant submersion with the
slant angles $\{ \theta_I = \frac{\pi}{6}, \theta_J = \frac{\pi}{2},
\theta_K = \frac{\pi}{3} \}$ and dilation $\lambda = e^{7}$.
\end{example}

\section*{Acknowledgments}
This research was supported by Basic Science Research Program through
the National Research Foundation of Korea(NRF)
 funded by the Ministry of Education (NRF-2016R1D1A1B03930449).


\end{document}